\documentclass[12pt,reqno]{article}
\usepackage[usenames]{color}
\usepackage{amssymb}
\usepackage{graphicx}
\usepackage{amscd}

\usepackage[colorlinks=true,
linkcolor=webgreen,
filecolor=webbrown,
citecolor=webgreen]{hyperref}

\definecolor{webgreen}{rgb}{0,.5,0}
\definecolor{webbrown}{rgb}{.6,0,0}

\usepackage{color}
\usepackage{fullpage}
\usepackage{float}

\usepackage{graphics,amsmath,amssymb}
\usepackage{amsthm}
\usepackage{amsfonts}
\usepackage{latexsym}

\begin{document}

\vspace*{2.1cm}

\theoremstyle{plain}
\newtheorem{theorem}{Theorem}
\newtheorem{corollary}[theorem]{Corollary}
\newtheorem{lemma}[theorem]{Lemma}
\newtheorem{proposition}[theorem]{Proposition}
\newtheorem{obs}[theorem]{Observation}
\newtheorem{claim}[theorem]{Claim}

\theoremstyle{definition}
\newtheorem{definition}[theorem]{Definition}
\newtheorem{example}[theorem]{Example}
\newtheorem{remark}[theorem]{Remark}
\newtheorem{conjecture}[theorem]{Conjecture}
\newtheorem{question}[theorem]{Question}

\begin{center}

\vskip 1cm

{\Large\bf Efficient total facial coloring of combinatorial maps} %
\vskip 5mm
\large
Italo J. Dejter

University of Puerto Rico

Rio Piedras, PR 00936-8377

\href{mailto:italo.dejter@gmail.com}{\tt italo.dejter@gmail.com}
\end{center}

\begin{abstract} Let $1<k\in\mathbb{Z}$.
A total coloring of a connected simple regular graph $G$ via color set $]k]_0=\{0,1,\ldots, k\}$  
 is efficient if each  
 color yields an efficient dominating set, where efficient domination applies to each color class in $V(G)$.
 An efficient total coloring of $G$ via $[k]_0$ extends to an efficient total facial coloring $c$ of a map $M(G)$ of $G$ if it assigns a vertex-and-edge $k$-color set to the boundary cycle of each face, with the consequently missing color 
 assigned to the face itself, so that the two adjacent faces along each edge are assigned different colors. Focus is set upon planar and toroidal maps.  
Modifications of $c$ are also treated by removing a perfect dominating set or a 1-factor from $G$. \end{abstract}


\section{Introduction}

In this paper, we combine the concepts of total coloring of graphs, efficient dominating sets and face coloring of combinatorial maps, with the accompaniment of confirming examples.

\subsection{Total coloring, efficient domination and combinatorial maps}

\begin{definition}\label{tc} Let $G$ be a connected simple graph. A {\it total coloring}, or {\it TC}, of $G$ is a color assignment of colors to $V(G)\cup E(G)$ such that adjacent vertices get different colors, edges incident to the same vertex get different colors  and  if an edge $e$ is incident with a vertex $v$ then $e$ and $v$ get different colors.
 The {\it total chromatic number} $\chi''(G)$ is the least number of colors required by a TC of $G$.\end{definition}

The TC Conjecture, posed by Behzad \cite{B1,B2} and Vizing \cite{V}, asserts that $\chi''(G)$ is either $\Delta(G)+1$ or $\Delta(G)+2$, where $\Delta$ is the largest degree of any vertex of $G$. 
 It was established for cubic graphs $G$ \cite{Feng,Mazzu,Rosen,Vi}, meaning that $\chi''(G)\in\{4,5\}$. To decide whether a cubic graph $G$ has $\chi''(G)=\Delta(G)+1$, even for bipartite cubic graphs, is NP-hard \cite{Arroyo}. A 2023 survey \cite{tc-as} contains an updated bibliography on TCs. 
 
In \cite{+1}, total colorings of $k$-regular graphs of girth $k+1$ ($1<k\in\mathbb{Z}$) in the presence of efficient dominating sets \cite{D73,Tomai} were considered. 

\begin{definition}\label{antes} 
A subset $S$ of $V(G)$ is a {\it perfect dominating set}, or {\it PDS}, if every vertex outside $S$ has exactly one neighbor in $S$. If in addition $S$ is stable, then $S$ is also said to be an {\it efficient dominating set}, or {\it EDS}.
\end{definition}

\begin{definition}\label{ahora} Let $N_G[v]$ and $N_G(v)$ be the {\it closed} and {\it open neighborhoods} of $v\in V(G)$, respectively \cite{D73}.
A vertex and edge coloring $c$ of a connected simple $k$-regular graph $G$ ($2\le k\in\mathbb{Z}$) is said to be an {\it efficient total coloring} ({\it ETC}) if:
\begin{enumerate}
\item[\bf(a)] as in Definition~\ref{tc}, each $v\in V(G)$ and its neighbors are assigned by $c$ all the colors in $[k]_0=\{0,1,\ldots,k\}$ via a bijection $N_G[v]=N_G(v)\cup\{v\}\leftrightarrow[k]_0$;
\item[\bf(b)] $c$ partitions $V(G)$ into $k+1$ EDSs.
\end{enumerate}\end{definition}

\noindent Definition~\ref{ahora} implies that the total chromatic number of $G$ is $\Delta(G)+1$. 
 

\begin{definition}\label{genius}
The {\it genus} of a graph $G$ is the number of handles that must be added to the sphere in order to avoid edge crossings in any drawing of $G$ in the resulting oriented surface.
\end{definition}

\begin{definition}\label{belt} Let $G$ be a connected simple graph.  
Let $S$ be a connected closed oriented surface of the genus of $G$. An embedding of $G$ in $S$ yields a {\it map} $M(G)$ of $G$ in $S$,
so there is a partition of $S\setminus(V(G)\cup E(G))$ into open simply connected regions of $\mathbb{R}^2$ whose closures in $S$ are said to be the {\it faces} of $M(G)$.
A {\it belt} of $M(G)$ is a face boundary cycle. Let $3<\ell\in\mathbb{Z}$. An $\ell$-{\it belt} of $M(G)$ is a belt that has length $\ell$. 
Let $F(G)$ be the set of faces of $M(G)$.\end{definition}

\begin{definition}\label{que} Let $\sigma$ be the correspondence that assigns to a map $M$ its {\it 1-skeleton}, namely the graph spanned by its vertices and edges.
Let $M(G)$ be a map of a connected simple graph $G$, so that $\sigma(M(G))=G$. 
An ETC $c:G\rightarrow[k]_0$ extends to an {\it efficient total facial coloring}, or {\it ETFC} $c':M(G)\rightarrow[k]_0$ if: {\bf(a)}  $c=c'|_G$; {\bf(b)} 
the belt $\sigma(\Phi)$ of each face $\Phi$ of $M$ has color $c'(\Phi)\in[k]_0$ missing in $c(V(\sigma(\Phi))\cup E(\sigma(\Phi)))$,
so $\{c'(\Phi)\}\cup c(V(\sigma(\Phi))\cup E(\sigma(\Phi)))\subseteq[k]_0$; {\bf(c)} 
the two adjacent faces along any edge have different colors. If in addition, any two faces with a common vertex have different colors, an ETFC is said to be a {\it strict} ETFC, or {\it SETFC}.
\end{definition}

\begin{definition}\label{sig1}
 Let $S\subset V(G)$ be a PDS. Let $G_S=G-S$ be $(k-1)$-regular graph. Let $M(G_S)$ be the map obtained by removing $S$ from $M(G)$, so a face $\Phi_v\in F_S=F(G_S)\setminus F(G)$ is created in $M(G_S)$ by the removal of each $v\in S$, with the consequent absence in $F(G_S)$ of the faces of $F(G)$ containing $v$, while all other faces of $M(G)$ are kept in $M(G_S)$. An (S)ETFC $c$ of
 $M(G)$ gives place to a coloring $c_S$ of $M(G_S)$ that equals $c$ on the vertices, edges and faces of $F(G)\cap F(G_S)$ and in which each face $\Phi_v$ of $F_S$ gets the color $c(v)$, i.e. $c_S(\Phi_v)=c(v)$. In this case, we still say that $c_S$ is an (S)ETFC.
\end{definition}

\begin{definition}\label{sig2}
Let $H\subset E(G)$ be a 1-factor of $G$. Let $G_H=G-H$ be $(k-1)$-regular graph. Let $M(G_H)$ be obtained by removing $H$ from $M(G)$, so a face $\Phi_e\in F_H=F(G_H)\setminus F(G)$ is created in $M(G_H)$ by the removal of each $e\in H$, with the consequent absence in $F(G_H)$ of the faces of $F(G)$ containing $e$, while all other faces of $M(G)$ are kept in $M(G_H)$. An (S)ETFC $c$ of
 $M(G)$ gives place to a coloring $c_H$ of $M(G_H)$ that equals $c$ on the vertices, edges and faces of $F(G)\cap F(G_H)$ and in which each face $\Phi_e$ of $F_H$ gets either color $c_S(\Phi_e)=c(e)$, if $c(e)$ is not the color of a kept neighboring old face (in $F'=F(G)\setminus F(G_H)$), or the color of an incident face of $e$ in $G$, otherwise. In either case, we still say that $c_H$ is an (S)ETFC. 
 \end{definition}
 
 \subsection{Tilings and lattices}
 
\begin{definition}\label{sETFC}\cite[Section 1.3]{Branko}
A {\it tiling} $T$ is a covering of a surface using one or more geometric shapes called {\it tiles}, with neither overlaps nor gaps. 
An {\it edge} of $T$ is the intersection of two bordering tiles. A {\it vertex} of $T$ is the nonempty intersection of three or more edges.  
Tiles  of $T$ are taken as closed sets that include vertices and edges.
A tiling $T$ is {\it regular} if its automorphism group acts transitively on its {\it flags}, each consisting of a vertex, an edge and a face mutually incident to each other.   
A tiling $T$ is {\it Archimedean} if the arrangements of tiles around all vertices are similar. A tiling $T$ is {\it uniform} if the automorphism group of $T$ acts transitively on the vertex set of $T$.
\end{definition}

\begin{definition} \cite[pg. 319]{BV}. 
Let $1<h\in\mathbb{Z}$. An {\it $h$-zonogon} is a convex polygon $$(A_0A_1A_2\cdots A_hA_{h+1}A_{h+2}\cdots,A_{h-1})$$ in $\mathbb{R}^2$
with $2h$ vertices $A_0,A_1,\ldots,A_{2h-1}$ and $2h$ sides given by segments $$A_0A_1,\, A_1A_2,\, \ldots,\, A_hA_{h+1},\, A_{h+1}A_{h+2},\, \ldots,\, A_{2h-1}A_0$$ forming pairs 
$$\{A_0A_1,A_{h+1}A_h\},\, \{A_1A_2,A_{h+2}A_{h+1}\},\, \ldots,\, \{A_{h-1}A_h,A_0A_{2h-1}\}$$ of parallel sides of equal lengths and opposite orientations. Given an $h$-zonogon $X$, an orientable surface $S=S_X$ of genus $h-1$ is obtained from $X$ by identifying its pairs of opposite sides:
\begin{eqnarray}\label{equiv}A_0A_1\equiv A_{h+1}A_h,\; A_1A_2\equiv A_{h+2}A_{h+1},\; \ldots,\; A_{h-1}A_h\equiv A_0A_{2h-1},\end{eqnarray} in which case $X$ is said to be an {\it $h$-cutout} of $S_X$. A map $M(G)$ of a simple connected graph $G$ of genus $h-1$ in $S_X$ can be represented by considering $X$ not only as an $h$-cutout of $S(X)$ but also as an {\it $h$-cutout} of $M(G)$. In such a case, $M(G)$ can be recovered by identifying the pairs of opposite sides of $X$, as expressed in (\ref{equiv}).
\end{definition}

\begin{definition}
A {\it lattice} in $\mathbb{R}^2$ is an infinite set of points
$\mathcal{L}=\{a_1v_1+a_2; a_1,a_2\in\mathbb{Z}\}$, where $\{v_1,v_2\}$ is a basis for $\mathbb{R}^2$.
The area enclosed by such vectors $v_1$ and $v_2$ forms a parallelogram (or rectangle) known as {\it fundamental domain}. 
A {\it lattice tiling} $T(\mathcal{L})$ is a tiling by translates of a single tile (e.g. a fundamental domain) by the points of a lattice $\mathcal{L}$. This tile serves as a 2-cutout for the toroid $S$ and the translates are said to be {\it lattice-translates}. Sub-lattices of such $\mathcal {L}$ are said to be sub-lattices of the corresponding $T(\mathcal{L})$.
\end{definition}

\begin{example}
A {\it rectangular cutout} of $S$ is a 2-cutout given as a rectangle $(ABCD)$ in $\mathbb{R}^2$ with opposite oriented side pairs $(AB,CD)$ and $(AC,BD)$ from which $S$ can be recovered via identifications $AB\equiv CD$ and $AC\equiv BD$. An {\it hexagonal cutout} of $S$ is a 3-cutout given as a convex 3-zonogon $(ABCDEF)$ with opposite oriented side pairs $(AB,DE)$, $(BC,EF)$ and $(CD,FA)$ from which $S$ can be recovered by means of the identifications $AB\equiv DE$, $BC\equiv EF$ and $CD\equiv FA$. 
Both rectangular and hexagonal cutouts $X$ may be considered as tiles of respective lattice tilings of $\mathbb{R}^2$ yielding the toroid $\mathbb{T}$ by identification of the pairs of opposite sides. Then, $\mathbb{T}$ can be considered as a topological quotient of the {\it covering map} $\theta:\mathbb{R}^2\rightarrow\mathbb{T}$ that sends $X$ and each of its lattice-translates onto $\mathbb{T}$ by the indicated identifications in $X$. 
\end{example}

\begin{remark}
Given a map $M(G)$ of a connected simple graph $G$ represented in a rectangular or hexagonal cutout $X$ or in its extension in $\mathbb{R}^2$ via the inverse image $\theta^{-1}(\mathbb{T})$, the edges of $G$ are represented by their endvertex pairs; each face $\Phi$ of $M(G)$ is represented by the set $V(\Phi)$ of its vertices ordered in one of the directions of its belt $\sigma(\Phi)$. As all the edges of 
 maps below are represented by segments of $\mathbb{R}^2$ and all faces of maps are represented by convex sets of $\mathbb{R}^2$, then w.l.o.g. we consider the edges and faces as compact subsets of $\mathbb{R}^2$, so that $\mathbb{R}^2$ and its quotient toroid $\mathbb{T}$ via $\theta$ are taken as the disjoint unions of the vertex points, the linear interiors of the edges and the interiors of the faces, where the {\it linear interior} of a segment $[(u,v),(z,w)]\in\mathbb{R}^2$ is $((u,v),(z,w))\in\mathbb{R}^2$, where $u,v,z,w\in\mathbb{R}$.
\end{remark}

\subsection{Plan of Sections 2 to 6}

In the following sections, (S)ETFCs on planar and toroidal maps are constructed. They arise from Euclidean tilings by convex regular polygons, namely the regular polygons (squares, triangles and hexagons) via Definition~\ref{que}, and Archimedean uniform tilings \cite{Branko}. It is shown that seven of the eight Archimedean uniform tilings of $\mathbb{R}^2$ admit (S)ETFCs via the removal of PDSs or 1-factors from the square and triangular tilings, as in Definitions~\ref{sig1}-\ref{sig2} and using Cundy-Rollet notation \cite{Cundy}.

\begin{figure}[htp]
\includegraphics[scale=1.47]{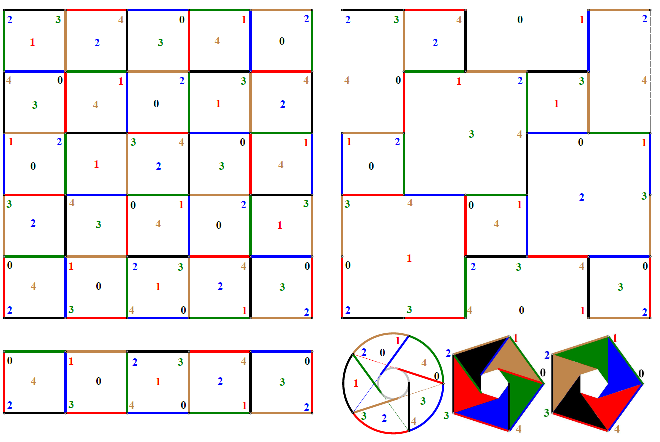}
\caption{SETFCs for the square and C\&R $4.8^2$ tilings. Toroidal SETFC via square tiling.}
\label{holes}
\end{figure}   

There are either six equilateral triangles or four squares or three regular hexagons around each vertex of a regular tiling of the plane, yielding (also in Cundy-Rosset notation \cite{Cundy}) the three known regular tiling, namely: {\bf(a)} the square tiling, C\&R $4^4$; {\bf(b)} the triangular tiling, C\&R $3^6$; {\bf(c)} the hexagonal tiling, C\&R $6^3$. These tilings are respectively inspected for the existence of (S)ETFCs in plane and toroidal maps in: {\bf(a)} Sections~\ref{sq} and~\ref{tri}  (Figures~\ref{holes} for $k=5$ and Figure~\ref{mod7} for $k=7$); {\bf(b)} Section~\ref{tri} for $k=7$ (left of Figure~\ref{cubo}); {\bf(c)} Section~\ref{ss3} for $k=7$ (Figure~\ref{tubo} and left of Figure~\ref{hubo}). 

As derivations of the built SETFCs from these three regular tilings, we inspect the eight Archimedean, uniform, tilings of the plane \cite{Grun}, whose Cundy and Rosset notations are: C\&R $4.8^2$ (right of Figure~\ref{holes}), C\&R $3^34^2$ (right of Figure~\ref{cubo}),  C\&R $3^2.4.3.4$ (Figure~\ref{7-5}), C\&R $(3.6)^2$ (right of Figure~\ref{hubo}), C\&Rs $3^4.6$ and 3.4.6.4 (Figure~\ref{3464}) and C\&Rs $3.12^3$ and failing $4.6.12$ (Figure~\ref{3.12^3}).
In addition, Figure~\ref{rombos}, dealing with the rhombille tiling \cite{Grun}, is also inspected.   

\section{SETFCs based on the square tiling}\label{sq}

In this section, we determine SETFCs based on the square tiling of $\mathbb{R}^2$ whose vertex set is the lattice $\mathbb{Z}^2$.

\begin{theorem}
There are SETFCs on planar and toroidal maps based on the square tiling of $\mathbb{R}^2$ by rectangular tiles that are interior-disjoint unions of translated copies of $[0,5]\times[0,1]$, including their continuation, or concatenation, over all of $\mathbb{R}^2$. Let $M'_{[]}$  be the smallest such toroidal map. Then $M'_{[]}$ has as its 1-skeleton  $\sigma(M'_{[]})$ the graph $\sigma(M'_{[]})=K_5$ obtained by identifying the left and right sides of the rectangle $[0,5]\times[0,1]$ as well as its bottom and top sides but with the top side translated two units to the right with respect to the bottom side. Moreover, if $d_{[]}$ is an SETFC of $M'_{[]}$, then each of $V(M'_{[]})$, $E(M'_{[]})$ and $F(M'_{[]})$ use just once each of the colors in $[5]$ in $d_{[]}$.
\end{theorem}

\begin{proof} Consider the map $M_\square$ provided in $\mathbb{R}^2$ with 
$$\begin{array}{ll}
\mathbb{Z}^2=&V(M_\square)=\{\;(x,y); x,y\in\mathbb{Z}\}\subset\mathbb{R}^2,\\
 &E(M_\square)=\{\{(x,y),(x+1,y)\};x,y\in\mathbb{Z}\}\cup\{\{(x,y),(x,y+1)\};x,y\in\mathbb{Z}\}\mbox{ and }\\
 &F(M_\square)=\{\{(x,y),(x+1,y),(x,y+1),(x+1,y+1)\};x,y\in\mathbb{Z}\}.
\end{array}$$
The faces in $F(M+\square)$ have interiors 
 $$Int([(x,y),(x+1,y)]\times[(x,y),(x,y+1)])=((x,y),(x+1,y))\times((x,y),(x,y+1)),$$
 for all $(x,y)\in\mathbb{Z}^2$. The edges in $E(M_\square)$ have interiors
$$Int([(x,y),(x+1,y)])=((x,y),(x+1,y))\mbox{ and }Int([(x,y),(x,y+1)])=((x,y),(x,y+1)),$$ for all $(x,y)\in\mathbb{Z}^2$. 
   
An SETFC $c_\square$ over $M_\square$ is given periodically over $\mathbb{R}^2$ by setting: 
\begin{eqnarray}\label{sq1}\begin{array}{l}
\;c_\square(x,y)=(2+x+2y)\mod 5,\mbox{ for }(x,y)\in V(M_\square)=\mathbb{Z}^2,\\  
c_\square(\{(x,y),(x+1,y)\})=(x+2y)\mod 5,\mbox{ for }\{(x,y),(x+1,y)\}\in E(M_\square),\\
c_\square(\{(x,y),(x,y+1)\})=(3+x+2y)\mod 5,\mbox{ for }\{(x,y),(x,y+1)\}\in E(M_\square),\\
c_\square(\{(x,y),(x+1,y),(x,y+1),(x+1,y+1)\})=(1+x+2y)\mod 5,\\
\;\mbox{ for }\{(x,y),(x+1,y),(x,y+1),(x+1,y+1)\}\in F(M_\square).\\
\end{array}\end{eqnarray}
This SETFC $c_\square$ is doubly periodic, with both horizontal and vertical periods equal to $5$. A double periodic tile for $c_\square$ in $\mathbb{R}^2$ is represented in display (\ref{oct6}), where the coordinate $x$ increases to the right and the coordinate $y$ increases downward, with the origin $O=(0,0)$ placed in the upper-left corner. The upper left of Figure~\ref{holes} represents such a periodic tile with its upper-left corner as the origin, where colors from 0 to 4 are given by colors black, red, blue, green and hazel, respectively.

\begin{eqnarray}\label{oct6}\begin{array}{ccc}
2\hspace*{3.0mm} _-^0\hspace*{3.0mm}3\hspace*{3.0mm} _-^1\hspace*{3.0mm}4\hspace*{3.0mm} _-^2\hspace*{3.0mm}0\hspace*{3.0mm}_-^3\hspace*{3.0mm}1\hspace*{3.0mm}_-^4\hspace*{3.0mm}2
\\
_3|\hspace*{3.0mm}1\hspace*{3.0mm}_4|\hspace*{3.0mm}2\hspace*{3.0mm}_0|\hspace*{3.0mm}3\hspace*{3.0mm}_1|\hspace*{3.0mm}4\hspace*{3.0mm}_2|\hspace*{3.0mm}0\hspace*{3.0mm}|_3
\\
4\hspace*{3.0mm} _-^2\hspace*{3.0mm}0\hspace*{3.0mm} _-^3\hspace*{3.0mm}1\hspace*{3.0mm} _-^4\hspace*{3.0mm}2\hspace*{3.0mm}_-^0\hspace*{3.0mm}3\hspace*{3.0mm}_-^1\hspace*{3.0mm}4
\\
_0|\hspace*{3.0mm}3\hspace*{3.0mm}_1|\hspace*{3.0mm}4\hspace*{3.0mm}_2|\hspace*{3.0mm}0\hspace*{3.0mm}_3|\hspace*{3.0mm}1\hspace*{3.0mm}_4|\hspace*{3.0mm}2\hspace*{3.0mm}|_0
\\
1\hspace*{3.0mm} _-^2\hspace*{3.0mm}2\hspace*{3.0mm} _-^0\hspace*{3.0mm}3\hspace*{3.0mm} _-^1\hspace*{3.0mm}4\hspace*{3.0mm}_-^2\hspace*{3.0mm}0\hspace*{3.0mm}_-^3\hspace*{3.0mm}1
\\
_2|\hspace*{3.0mm}0\hspace*{3.0mm}_3|\hspace*{3.0mm}1\hspace*{3.0mm}_4|\hspace*{3.0mm}2\hspace*{3.0mm}_0|\hspace*{3.0mm}3\hspace*{3.0mm}_1|\hspace*{3.0mm}4\hspace*{3.0mm}|_2
\\
3\hspace*{3.0mm} _-^1\hspace*{3.0mm}4\hspace*{3.0mm} _-^2\hspace*{3.0mm}0\hspace*{3.0mm} _-^3\hspace*{3.0mm}1\hspace*{3.0mm}_-^4\hspace*{3.0mm}2\hspace*{3.0mm}_-^0\hspace*{3.0mm}3
\\
_4|\hspace*{3.0mm}2\hspace*{3.0mm}_0|\hspace*{3.0mm}3\hspace*{3.0mm}_1|\hspace*{3.0mm}4\hspace*{3.0mm}_2|\hspace*{3.0mm}0\hspace*{3.0mm}_3|\hspace*{3.0mm}1\hspace*{3.0mm}|_4
\\
0\hspace*{3.0mm} _-^3\hspace*{3.0mm}1\hspace*{3.0mm} _-^4\hspace*{3.0mm}2\hspace*{3.0mm} _-^0\hspace*{3.0mm}3\hspace*{3.0mm}_-^1\hspace*{3.0mm}4\hspace*{3.0mm}_-^2\hspace*{3.0mm}0
\\
_1|\hspace*{3.0mm}4\hspace*{3.0mm}_2|\hspace*{3.0mm}0\hspace*{3.0mm}_3|\hspace*{3.0mm}1\hspace*{3.0mm}_4|\hspace*{3.0mm}2\hspace*{3.0mm}_0|\hspace*{3.0mm}3\hspace*{3.0mm}|_1
\\
2\hspace*{3.0mm} _-^0\hspace*{3.0mm}3\hspace*{3.0mm} _-^1\hspace*{3.0mm}4\hspace*{3.0mm} _-^2\hspace*{3.0mm}0\hspace*{3.0mm}_-^3\hspace*{3.0mm}1\hspace*{3.0mm}_-^4\hspace*{3.0mm}2
\\
\end{array}\end{eqnarray}
In other words, since the vertex set $V(M_\square)=\mathbb{Z}^2$ is the plane integer lattice, the upper left of Figure~\ref{holes} or display (\ref{oct6}) shows that, by identifying in parallel the left and right sides of the tile $\mathbb{L}=[0,5]\times[0,5]$ and then its top and bottom sides, a toroid $\mathbb{T}_{5\times 5}$ is obtained via a quotient map $\phi:\mathbb{R}^2\rightarrow\mathbb{T}_{5\times 5}$ taking $M_\square$ and its SETFC $c_\square$ onto a toroidal map $M'_\square$ and corresponding toroidal SETFC $d_\square$, respectively.

The tile $\mathbb{L}$ belongs to the tiling $\mathcal{L}_{5\times 5}=\{[5p,5p+5]\times[5q,5q+5];(p,q)\in\mathbb{Z}^2\}$ of $\mathbb{R}^2$ based on the sub-lattice $\{(5p,5q);p,q\in\mathbb{Z}\}$ of $\mathbb{Z}^2$.
We note that other toroidal maps and SETFCs are obtained from sub-lattices of $\mathbb{Z}^2$ that are coarser than $\mathcal{L}_{5\times 5}$, with rectangular tiles that are the union of tiles of $\mathcal{L}_{5\times 5}$.  

Still other toroidal maps are obtained from $M'_\square$ by passing it through doubly periodic quotient maps from $\mathbb{R}^2$ onto rectangular tiles $[0,5]\times[0,j]$, where $1\le j<5$, with identification in parallel of the vertical sides and displaced identification of the top and bottom sides. The smallest tile of an adequate sub-tiling $\mathcal{L}_{1\times 5}$ here is formed say covering the rectangle $\mathbb{L}=[0,5]\times[4,5]$ (see lower-left of Figure~\ref{holes}), which yields the toroidal map $M'_{[]}$ having $\sigma(M'_{[]})=K_5$ and SETFC $c_{[]}$ as in the cutout in the lower right of Figure~\ref{holes}, with the face colors shown schematically to its right on the visible convex side of a toroid and farther to the right on its corresponding hidden concave side.
\end{proof}

A particular case of a construction of \cite{Dean}: Let $K_4(K_5)$ be the graph whose vertices are the arcs of $K_5$ such that for each $v\in K_5$ there is a copy of $K_4$ whose vertices are the arcs departing from $v$ and such that for each edge $\{v,w\}$ of $K_5$ there is an edge $\{v,w\}$ of $K_4(K_5)$ joining the vertices representing the arcs $(v,w)$ and $(w,v)$. We visualized $K_4(K_5)$ as obtained from $K_5$ by replacing each vertex by a copy of $K_4$. 

Consider $K_5$ as an edge-colored graph $K_5=C_5^r\cup C_5^b$ with $V(C_5^r)=V(C_5^b)=V(K_5)$, where $C_5^r$ and $C_5^b$ are 5-cycles with the edges red and blue, respectively. Assigning to the two arcs representing an edge $e$ of $K_5$ the same color of $e$, we assign to the vertices of $K_4(K_5)$ the colors of the arcs they are representing.
Then, let $K_4'(K_5)=K_4'(C_5^r\cup C_5^b)$ be obtained from $K_4(K_5)=K_4(C_5^r\cup C_5^b)$ by removing the pair of edges in each copy of $K_4$ representing a vertex of $K_5$ with different endvertex colors. The upper-right of Figure~\ref{holes} represents a toroidal cutout $\mathbb{D}$ of $K_4'(K_5)$ as in the following corollary. 
Recall that $P_n$ denotes an $(n-1)${\it -path}, meaning a path of length $n-1$, and that $P_m\square P_n$ stands for the {\it cartesian product} \cite{Imrich} of the paths $P_m$ and $P_n$. The segment $[0,5]\subset\mathbb{R}$ is taken as a copy of the path  $P_6$ in the graph induced by the lattice $\mathbb{Z}$ in $\mathbb{R}$.

\begin{corollary}
There are SETFCs in plane and toroidal maps based on the Archimedean uniform {\rm C\&R} $4.8^2$ tiling of $\mathbb{R}^2$ by rectangular tiles that are the interior-disjoint unions of translated copies of $[0,5]\times[0,5]$ and their continuation, or concatenation, over all of $\mathbb{R}^2$. Let $M'_{=}$  be the smallest such map. Then, $\sigma(M'_{=})=K_4'(K_5)$ is the toroidal graph obtained as the quotient of $[0,5]\square[0,5]\setminus\{(0,1),(3,0),(2.2),(5,1),(1,4),(3,5)\}$ by identifying the left and right sides of the rectangular cutout $[0,5]\times[0,5]$ as well as its bottom and top sides. 
\end{corollary}

\begin{proof}
The upper-right of Figure~\ref{holes} represents the toroidal cutout $\mathbb{D}$ of $M'_{=}$ obtained from $[0,5]\times[0,5]$ by deleting the EDS $S$ of vertices $(x,y)=(0,1),(3,0),(2.2),(5,1),(1,4),$ $(3,5)$, where the coordinate $x$ increases to the right and the coordinate $y$ increases downward, with the origin $O=(0,0)$ placed in the upper-left corner. Clearly, the shown restriction of $c_\square$ in the figure defines $d_{=}$. All rectangular tiles obtained by the interior-disjoint union of translated copies of $\mathbb{D}$ are the remaining toroidal cutouts in the statement producing SETFCs  on corresponding toroidal maps. By continuation of this process, in the limit an infinite 
 map covering the whole plane is obtained that contains a corresponding SETFC.  
\end{proof}

\section{(S)ETFCs based on the triangular tiling}\label{tri}

\begin{theorem}\label{tdelta}
There are SETFCs in plane and toroidal maps based on the triangular tiling of $\mathbb{R}^2$ via rectangular tiles that are copies of  $[0,7p]\times[0,q]$, where $0<p,q\in\mathbb{Z}$, each such copy taken as a map obtained from $P_{7p+1}\square P_{q+1}$ by adding the main-diagonal edges of its 4-cycle squares, including their continuation, or concatenation, over all of $\mathbb{R}^2$. 
Let  $M'_\Delta$ be the smallest such map. Then $\sigma(M'_\Delta)=K_7$ is the toroidal graph obtained by identifying the left and right sides of the rectangular cutout $[0,7]\times[0,1]$ as well as its bottom and top sides but with its top side translated two units to the right with respect to its bottom side. Moreover, if $d_\Delta$ is an SETFC of $M'_\Delta$, then each of $V(M'_\Delta)$, $E(M'_\Delta)$ and $F(M'_\Delta)$ use once each of the colors of the set $[7]$ in $d_\Delta$.
\end{theorem}

\begin{proof}
Consider the map $M_\Delta$ provided in $\mathbb{R}^2$ with
 
$$\begin{array}{ll}
V(M_\Delta)\!=\{(x,y); x,y\in\mathbb{Z}\}\subset\mathbb{R}^2,\\
 E(M_\Delta)\!=\{\{(x,y),(x+1,y)\};x,y\in\mathbb{Z}\}\cup&\{\{(x,y),(x,y+1)\};x,y\in\mathbb{Z}\}\cup\\
 &\{(x+1,y),(x,y+1)\};x,y\in\mathbb{Z}\}\mbox{ and }\\
 F(M_\Delta)\!=\{\{(x,y),(x+1,y),(x,y+1)\}\cup&\{(x+1,y),(x,y+1),(x+1,y+1)\};\!x,\!y\!\in\mathbb{Z}\},\\
\end{array}$$
where $E(M_\Delta)$ is formed by the horizontal and vertical sides and the main-diagonals of the unit squares $[(x,y),(x+1,y)]\times[(x,y),(x,y+1)]$, with $(x,y)\in\mathbb{Z}$. 

 An SETFC $c_\Delta$ in $M_\Delta$ is given periodically by setting
\begin{eqnarray}\label{sq2}\begin{array}{l}
c_\Delta(x,y)=(x+5y)\mod 7,\mbox{ for }(x,y)\in V(M_\Delta)=\mathbb{Z}^2,\\  
c_\Delta(\{(x,y),(x+1,y)\})=(4+x+5y)\mod 7,\mbox{ for }\{(x,y),(x+1,y)\}\in E(M_\Delta),\\
c_\Delta(\{(x,y),(x,y+1)\})=(6+x+5y)\mod 7,\mbox{ for }\{(x,y),(x,y+1)\}\in E(M_\Delta),\\
c_\Delta(\{(x,y),(x+1,y+1)\})=(3+x+5y)\mod 7,\\
\;\;\mbox{ for }\{(x,y),(x+1,y+1)\}\in E(M_\Delta),\\
c_\Delta(\{(x,y),(x+1,y),(x,y+1)\})=(2+x+5y)\mod 7,\\\;\;\mbox{ for }\{(x,y),(x+1,y),(x,y+1)\}\in F(M_\Delta).\\
c_\Delta(\{(x+1,y),(x,y+1),(x+1,y+1)\})=(4+x+5y)\mod 7,\\
\;\;\mbox{ for }\{(x+1,y),(x,y+1),(x+1,y+1)\}\in F(M_\Delta).\\
\end{array}\end{eqnarray}
This SETFC $c_\Delta$ is doubly periodic, illustrated on the left of Figure~\ref{cubo}, where coordinate axes are horizontal for $x$ and tilted at an angle of 60 degrees or $\frac{\pi}{2}$ radians (said to be ``anti-diagonal") for $y$, with both horizontal and tilted periods equal to $7$. 
Colors in such representation of $c_\Delta$ are given orderly from 0 to 6 as black, red, blue, green, hazel, violet and rose, respectively.
A double periodic tile of $c_\Delta$ in $M_\Delta$ is represented in display (\ref{dec26}), where coordinate $x$ increases to the right and coordinate $y$ increases upward, with the origin $O=(0,0)$ taken as the leftmost vertex.

\begin{eqnarray}\label{dec26}\begin{array}{c}
5\hspace*{3.4mm} _-^2\hspace*{3.4mm}6\hspace*{3.4mm} _-^3\hspace*{3.4mm}0\hspace*{3.4mm} _-^4\hspace*{3.4mm}1\hspace*{3.4mm}_-^5\hspace*{3.4mm}2\hspace*{3.4mm}_-^6\hspace*{3.4mm}3
\hspace*{3.4mm} _-^0\hspace*{3.4mm}4\hspace*{3.4mm} _-^1\hspace*{3.4mm}5
\\
_6|\hspace*{1.4mm}_23^4\hspace*{1.4mm}_0|\hspace*{1.4mm}_34^5\hspace*{1.4mm}_1|\hspace*{1.4mm}_45^6\hspace*{1.4mm}_2|\hspace*{1.4mm}_56^0\hspace*{1.4mm}_3|\hspace*{1.4mm}_60^1\hspace*{1.4mm}|_4
\hspace*{1.4mm}_01^2\hspace*{1.4mm}_5|\hspace*{3.0mm}_12^3\hspace*{1.5mm}_6|
\\
0\hspace*{3.4mm} _-^4\hspace*{3.4mm}1\hspace*{3.4mm} _-^5\hspace*{3.4mm}2\hspace*{3.4mm} _-^6\hspace*{3.4mm}3\hspace*{3.4mm}_-^0\hspace*{3.4mm}4\hspace*{3.4mm}_-^1\hspace*{3.4mm}5
\hspace*{3.4mm} _-^2\hspace*{3.4mm}6\hspace*{3.4mm} _-^3\hspace*{3.4mm}0
\\
\end{array}\end{eqnarray}

\begin{figure}[htp]
\includegraphics[scale=0.58]{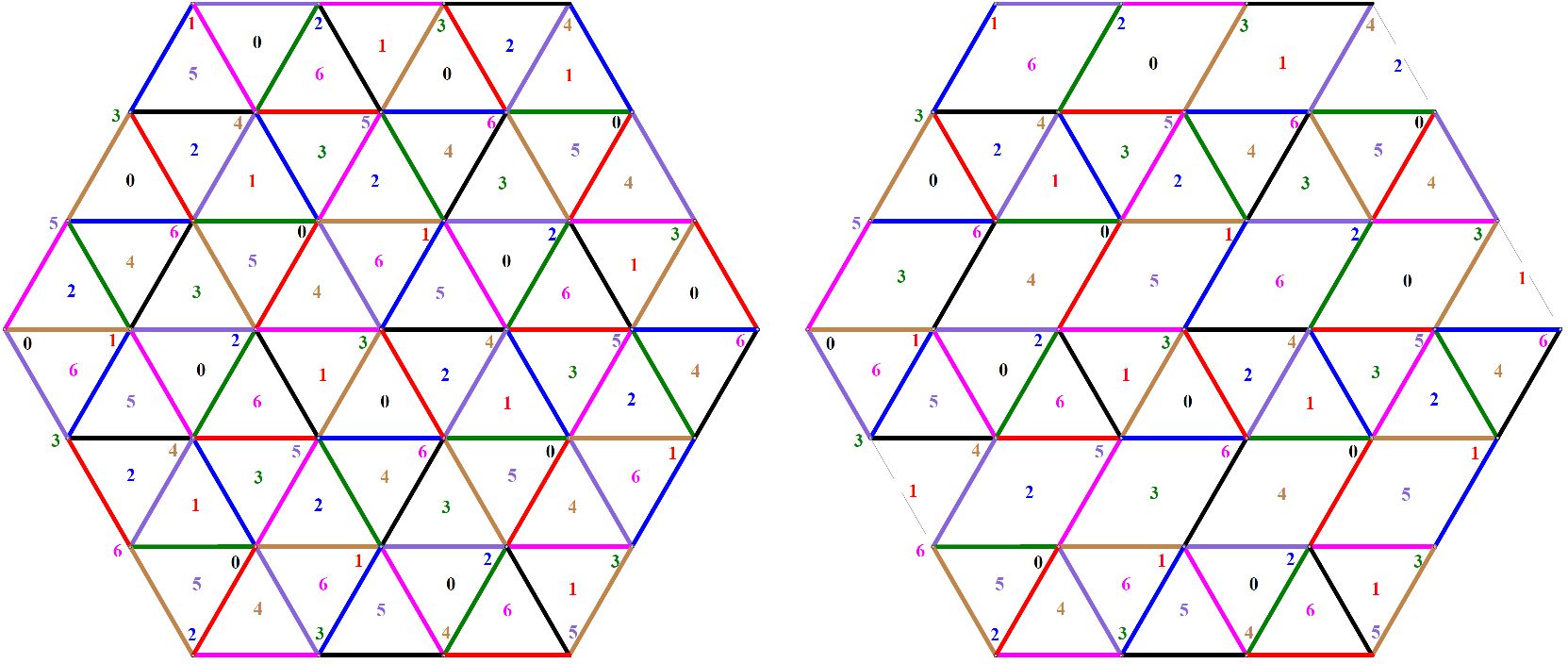}
\caption{SETFC for the triangular and C\&R $3^3.4^2$ tilings.}
\label{cubo}
\end{figure}

\begin{figure}[htp]
\includegraphics[scale=0.58]{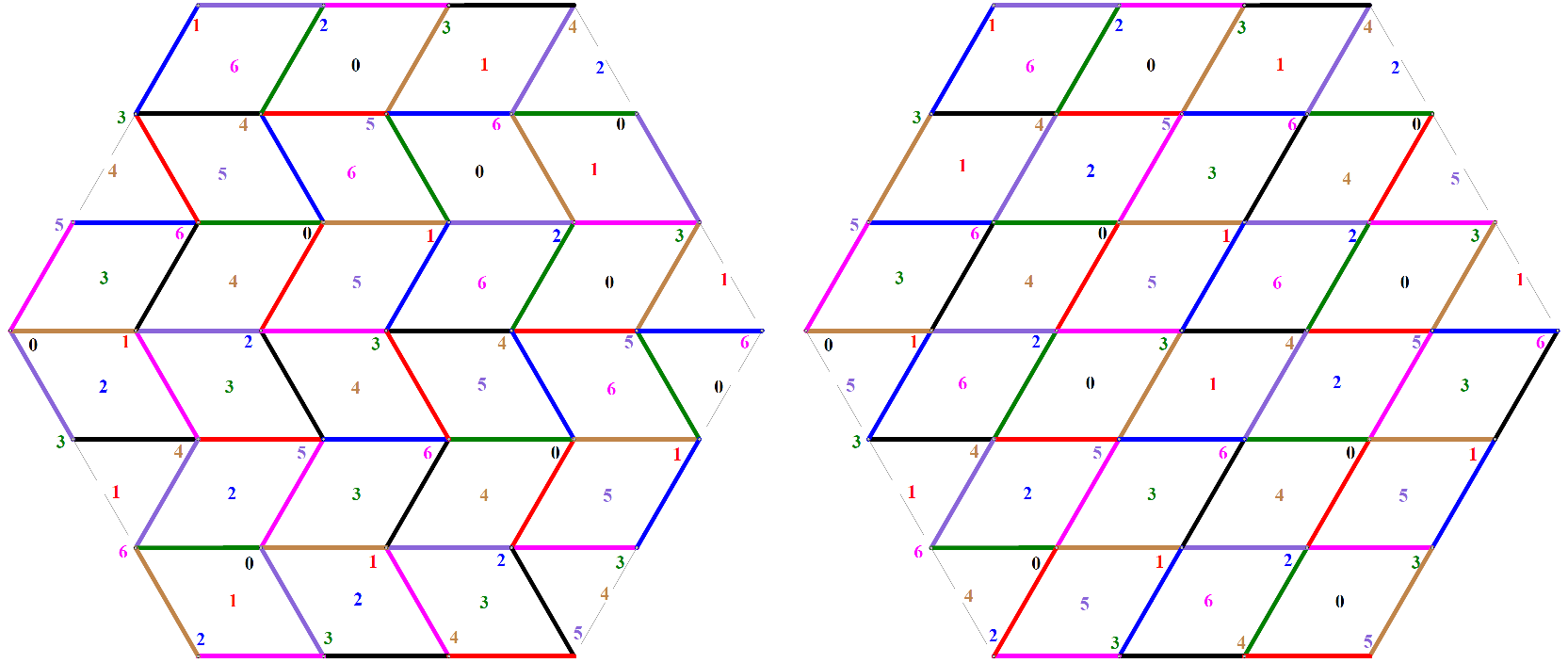}
\caption{ETFCs mod 7 for square tiling based on the triangular tiling.}
\label{mod7}
\end{figure}

Display (\ref{dec26}) shows a union $\mathcal{T}$ of fourteen faces of a map of $\mathbb{R}^2$ obtained by double periodic translations of $\mathcal{T}$. 
Each of the seven 4-cycles of $\mathcal{T}$ formed by two vertical edges and two horizontal edges is assumed divided into two triangles separated by the main diagonal taken as an edge represented just by its color number $\xi$; this $\xi$ together with the four vertex color numbers are shown larger than the four edge color numbers, while the color numbers of the two triangles separated by $\xi$ appear as a pre-subscript and a post-superscript of such $\xi$.
Note that $\mathcal{T}$ can be taken as a rectangular cutout of the $M'_\Delta$ in the statement in the following sense and bearing the corresponding $c_\Delta$.  
It is easy to see that $\mathcal{T}$ is equivalent to an horizontal strip of fourteen triangles in the extension of $\mathcal{T}$ to all of $\mathbb{R}^2$, partially seen on the left of Figure~\ref{cubo}.
By identifying the left and right sides of $\mathcal{T}$ and then its top and bottom sides, with the top side displaced two units to the right with respect to the bottom side, $M'_\Delta$ and $c_\Delta$ are obtained in the toroid $\mathbb{T}$ via the corresponding quotient map $\phi_\Delta:\mathbb{R}^2\rightarrow\mathbb{T}$. 
\end{proof}

\begin{corollary}\label{coro}
There are SETFCs in plane and toroidal maps based on the Archimedean uniform {\rm C\&R} $3^3.4^2$ tiling of $\mathbb{R}^2$ via rectangular tiles that are the interior-disjoint unions of translated copies of $[0,7]\times[0,2]$ and their continuation, or concatenation, over all of $\mathbb{R}^2$. Let $M'_{\neq}$  be the smallest such map with corresponding SETFC $c_{\neq}$. Then, $\sigma(M'_{\neq})$ is obtained on the toroid by identifying the left and right sides of $[0,7]\times[0,2]\setminus\{\{(j,0),(j+1,1)\};j\in[7]\}$ as well as its bottom and top sides, with the top side displaced four units to the right with respect to the bottom side. 
\end{corollary}

\begin{figure}[htp]
\includegraphics[scale=1.48]{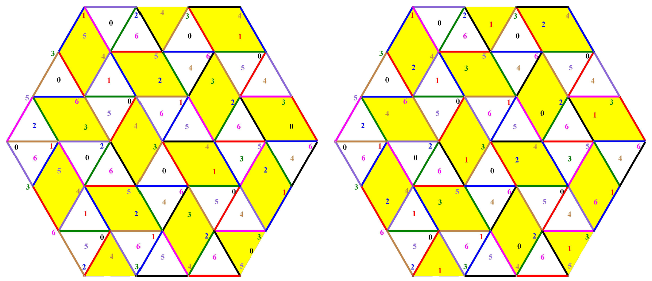}
\caption{Two SETFCs for C\&R $3^2.4.3.4$ tiling based on the triangular tiling.}
\label{7-5}
\end{figure}

\begin{proof}
The right of Figure~\ref{cubo} represents the toroidal cutout $\mathbb{D}'$ of $M'_{\neq}$ obtained from $[0,7]\times[0,2]$ by deleting the 1-factor $M=\{(j,0)(j+1,1);j\in[7]\}$. Clearly, the shown restriction of $c_\Delta$ in the figure defines $c_{\neq}$. All rectangular tiles obtained by the interior-disjoint union of translated copies $\mathbb{D}'$ are the remaining toroidal cutouts in the statement producing SETFCs $c_{\neq}$ on corresponding toroidal maps $M_{\neq}$. By continuation of this process, in the limit an infinite map covering the whole plane is obtained that contains a corresponding SETFC.  
\end{proof}

\begin{corollary}
From the SETFCs in plane and toroidal maps obtained in Corollary~\ref{coro}, new ETFCs (resp. SETFCs) with $k=7$ in plane and toroidal maps based on the square tiling seen as a rhombic tiling are obtained by deleting in the horizontal rows of triangles all the anti-diagonal (resp. main-diagonal) edges, letting each new rhombic face resulting from the said deletion receive the color of the corresponding deleted edge.
\end{corollary}

\begin{proof} Illustrated on the left of Figure~\ref{mod7}, the first instance in the statement has the face colors around each vertex having a common pair of opposite face colors, so to $M$ we add a 1-factor $M'$.
This case yields an ETFC, not an SETFC.  
Illustrated on the right of the figure, the second, parenthesized, instance has the colors of the four faces sharing any common vertex being pairwise distincto\, so to $M$ with add a 1-factor $M''$. This is an SETFC. 
\end{proof}

\begin{corollary}
From the SETFCs in plane and toroidal maps obtained in Theorem~\ref{tdelta}, new ETFCs with $k=7$ in 
plane and toroidal maps based on  the {\rm C\&R} $3^2.4.3.4$ Archimedean uniform tiling are obtained by deleting edges periodically mod 7 and selecting, for each new resulting rhombic face, one of the two old triangle color numbers as its new color.
\end{corollary}

\begin{proof} 
We remove a 1-factor $M'''$ from the map on the left of Figure~\ref{cubo} to get the maps on the left and right of  
Figure~\ref{7-5} on the claimed plane Archimedean uniform tiling. The figure shows two complementary instances of choosing in each resulting rhombic face one of the two old triangle color numbers in order to obtain an ETFC that is not SETFC. 
\end{proof}

\section{SETFCs based on the hexagonal tiling}\label{ss3}

The hexagonal tiling of $\mathbb{R}^2$ \cite{Branko} appears in the following context.
The {\it dual map} $M^*$ of a planar graph embedding $M$ is obtained by setting $V(M^*)$ as the set of faces of $M$ and an edge $e^*=\{f^*,g^*\}$ of $E(M^*)$ for each edge $e$ of $M$ adjacent to faces $f,g$ of $M$. For example, the triangular tiling yields a map $M$ whose vertex set is composed by the vertices of the triangular tiles and its edge set is composed by the edges of the triangular tiles. The dual graph $M^*$ of $M$ will be said to be {\it barycentric} if it can be seen (geometrically) as having the barycenters of the triangular tiles of $M$ as its vertices and the segments joining the barycenters of adjacent triangular tiles as its edges. Clearly, $M^*$ is the map of the hexagonal tiling \cite{Branko}

Now, let the toroid $\mathbb{T}$ be obtained by identifying in parallel the left and right sides of the rectangle $[0,14p]\times[0,2q]$ as well as its top and bottom sides, with the top side displaced two units to the right with respect to the bottom side. Let $\phi:\mathbb{R}^2\rightarrow\mathbb{T}$ be the double-periodic projection map of horizontal period 14 and vertical period 2. Given a map $M$ in $\mathbb{R}^2$ that is double periodic with horizontal period 14 and vertical period 2, its barycentric dual map $M^*$ becomes periodic in the same fashion as that of $M$. Then, the quotient map $\phi(M^*)$ is said to be the {\it toroidal dual graph} of $\phi(M)$.

\begin{theorem}\label{tnabla}
There are SETFCs in toroidal maps based on the hexagonal tiling of $\mathbb{R}^2$ via tiles that are copies of $[0,14p]\times[0,2q]$, where $0<p,q\in\mathbb{Z}$, each such copy taken as a map obtained from $P_{14p+1}\square P_{2q+1}$ by removing alternate vertical edges to yield 6-belts of horizontal length 2 and vertical length 1, including their continuation, or concatenation, over all of $\mathbb{R}^2$. The resulting maps $M_\nabla$ have graphs $\sigma(M_\nabla)$ that are cubic maps.  
The smallest such map $M'_\nabla$ is the toroidal dual map of the smallest $M'_\Delta$ in Theorem~\ref{tdelta}. Its 1-skeleton $\sigma(M'_\nabla)$ coincides with the Heawood graph.
\end{theorem}

\begin{proof}
The dual map $M_\nabla$ of $M_\Delta$ in the proof of Theorem~\ref{tdelta} provides an example of SETFC.
A partial superposed representation of both $M_\Delta$ and $M_\nabla$ is given on the left of Figure~\ref{tubo}, showing in $M_\nabla$ the SETFC $c=c_\nabla$ {\it dual} to the SETFC $c_\Delta$ of $M_\Delta$, where colors, of vertices, edges and faces of $c_\Delta$ become the colors of the corresponding faces, edges and vertices of $c_\nabla$, respectively. The right of Figure~\ref{tubo} represents just the part of $M_\nabla$ colored according to $c_\nabla$, as on the left of the figure. 

\begin{eqnarray}\label{dec6}\begin{array}{ccc}
2\hspace*{2.2mm}^-_3\hspace*{2.2mm}4\hspace*{2.2mm}^-_0\hspace*{2.2mm}
3\hspace*{2.2mm}^-_4\hspace*{2.2mm}5\hspace*{2.2mm}^-_1\hspace*{2.2mm} 
4\hspace*{2.2mm}^-_5\hspace*{2.2mm} 6\hspace*{2.2mm}^-_2\hspace*{2.2mm}
5\hspace*{2.2mm}^-_6\hspace*{2.2mm}0\hspace*{2.2mm}^-_3\hspace*{2.2mm}
6\hspace*{2.2mm}^-_0\hspace*{2.2mm}1\hspace*{2.2mm}^-_4\hspace*{2.2mm} 
0\hspace*{2.2mm}^-_1\hspace*{2.2mm} 2\hspace*{2.2mm}^-_5\hspace*{2.2mm}
1\hspace*{2.2mm}^-_2\hspace*{2.2mm} 3\hspace*{2.2mm}^-_6\hspace*{2.2mm}2
\\
\!_4|\hspace*{6.5mm}1\hspace{6.5mm}_5|\hspace*{6.5mm}2\hspace{6.5mm}_6|\hspace*{6.5mm}3\hspace{6.5mm}_0|\hspace*{6.5mm}4\hspace{6.5mm}_1|\hspace*{6.5mm}5\hspace{6.5mm}_2|\hspace*{6.5mm}6\hspace{6.5mm}_3|\hspace*{6.5mm}0\hspace{6.5mm}_4|\\
6\hspace*{2.2mm}^-_2\hspace*{2.2mm}5\hspace*{2.2mm}^-_6\hspace*{2.2mm}
0\hspace*{2.2mm}^-_3\hspace*{2.2mm}6\hspace*{2.2mm}^-_0\hspace*{2.2mm} 
1\hspace*{2.2mm}^-_4\hspace*{2.2mm} 0\hspace*{2.2mm}^-_1\hspace*{2.2mm}
2\hspace*{2.2mm}^-_5\hspace*{2.2mm}1\hspace*{2.2mm}^-_2\hspace*{2.2mm}
3\hspace*{2.2mm}^-_6\hspace*{2.2mm}2\hspace*{2.2mm}^-_3\hspace*{2.2mm} 
4\hspace*{2.2mm}^-_0\hspace*{2.2mm} 3\hspace*{2.2mm}^-_4\hspace*{2.2mm}
5\hspace*{2.2mm}^-_1\hspace*{2.2mm} 4\hspace*{2.2mm}^-_0\hspace*{2.2mm}6
\\
3\hspace{6.5mm}_0|\hspace*{6.5mm}4\hspace{6.5mm}_1|\hspace*{6.5mm}5\hspace{6.5mm}_2|\hspace*{6.5mm}6\hspace{6.5mm}_3|\hspace*{6.5mm}0\hspace{6.5mm}_4|\hspace*{6.5mm}1\hspace{6.5mm}_5|\hspace*{6.5mm}2\hspace{6.5mm}_6|\hspace*{6.5mm}3\\
0\hspace*{2.2mm}^-_1\hspace*{2.2mm}2\hspace*{2.2mm}^-_5\hspace*{2.2mm}
1\hspace*{2.2mm}^-_2\hspace*{2.2mm}3\hspace*{2.2mm}^-_6\hspace*{2.2mm} 
2\hspace*{2.2mm}^-_3\hspace*{2.2mm} 4\hspace*{2.2mm}^-_0\hspace*{2.2mm}
3\hspace*{2.2mm}^-_4\hspace*{2.2mm}5\hspace*{2.2mm}^-_1\hspace*{2.2mm}
4\hspace*{2.2mm}^-_5\hspace*{2.2mm}6\hspace*{2.2mm}^-_2\hspace*{2.2mm} 
5\hspace*{2.2mm}^-_6\hspace*{2.2mm} 0\hspace*{2.2mm}^-_3\hspace*{2.2mm}
6\hspace*{2.2mm}^-_0\hspace*{2.2mm} 1\hspace*{2.2mm}^-_4\hspace*{2.2mm}0 
\\
\end{array}\end{eqnarray}

\begin{figure}[htp]
\includegraphics[scale=0.67]{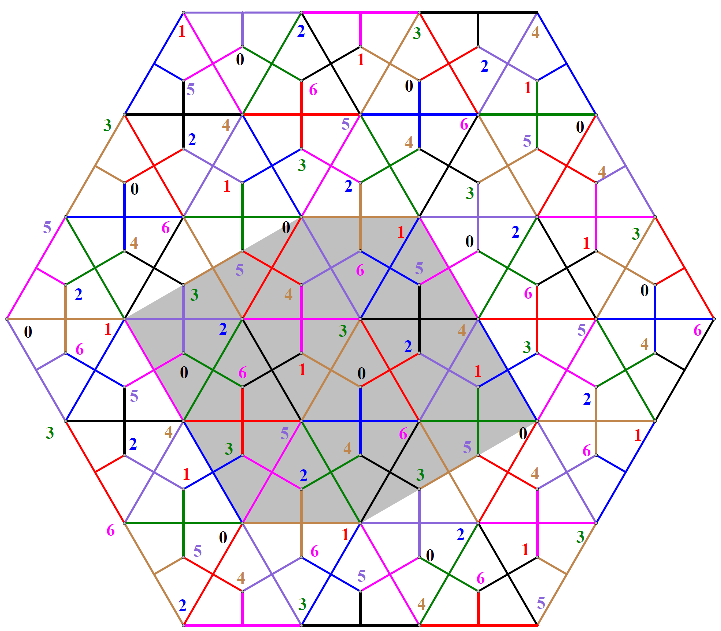}
\includegraphics[scale=0.56]{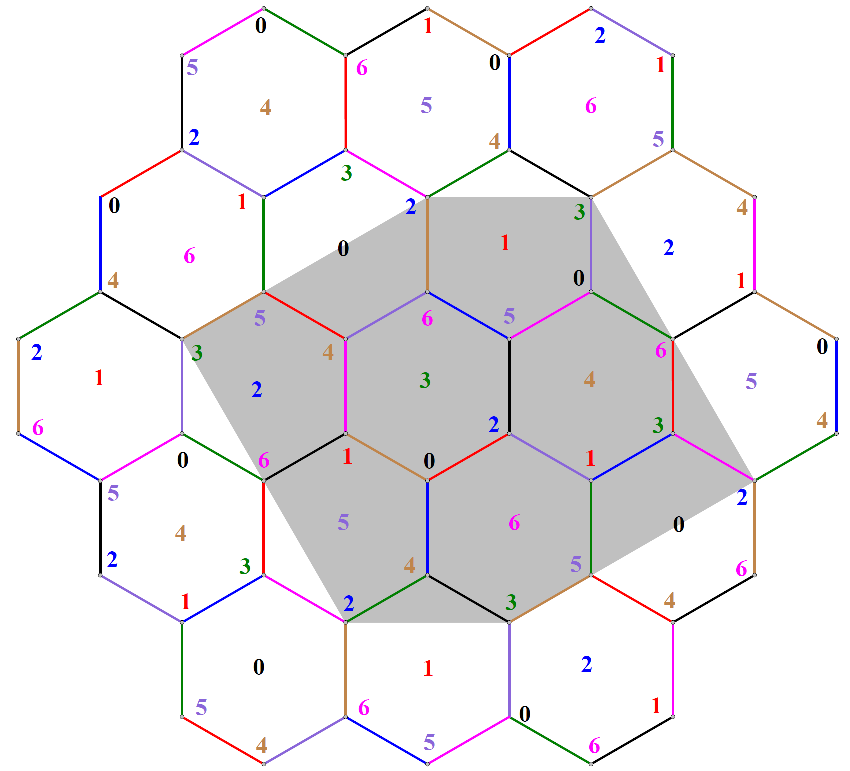}
\caption{SETFCs based on the triangular and hexagonal tilings.}
\label{tubo}
\end{figure}

Display (\ref{dec6}) represents a tile $\mathbb{L}=[0,14]\times[0,2]\subset\mathbb{R}^2$ in which $M_\nabla$ is represented, where the tilted and vertical edges in the right of Figure~\ref{tubo} are represented horizontally and vertically in $\mathbb{L}$, respectively, and vertex-and-face color numbers are shown larger than edge color numbers.
By identifying the left and right vertical sides of $\mathbb{L}$ in parallel and then identifying the horizontal top and bottom sides of $\mathbb{L}$, with the bottom side displaced four units to the right with respect to the top side, a toroid $\mathbb{T}$ is obtained, with $M_\nabla$ and $c_\nabla$ mapped via a corresponding quotient map  $\phi:\mathbb{R}^2\rightarrow\mathbb{T}$ onto a toroidal map $M'_\nabla$ and a corresponding quotient SETFC $d_\nabla$. 
Here, $M'_\nabla$ coincides clearly with the  Heawood graph and can be interpreted as toroidal dual of $M'_\Delta$, and the SETFC $d_\nabla$ can be interpreted as the toroidal dual of the SETFC $d_\Delta$.
\end{proof}

\begin{remark}
In each of both the left and right sides of Figure~\ref{tubo}, a region is shaded in light gray color representing a tile in the plane that yields, by adequate identification of opposite edges (see the paragraph following this remark), a representation of the toroid containing,  on the left side a detached superposed representation of both $(M'_\Delta,d_\Delta)$ and $(M'_\nabla,d_\nabla)$, and on the right side the sole detached representation of $(M'_\nabla,d_\nabla)$. This last representation yields the 7-coloring of the unit distance graph with vertex set $\mathbb{R}^2$, obtained by modifying adequately $d_\Delta(\sigma(M'_\Delta))$, namely the colors of the vertices and edges of $M'_\Delta$, which is used to establish 7 as the upper bound in the Hadwiger-Nelson problem \cite[pp. 150--152]{Jensen}, \cite{Soifer}.

To obtain $(M'_\nabla,d_\nabla)$ from $(M_\nabla,c_\nabla)$, note that on the border of the shaded area on the right of Figure~\ref{tubo} the following pairs of paths must be identified,  where subindex numbers represent edge colors between corresponding endvertex color numbers:
\begin{enumerate} 
\item the bottom and top left-to-right 3-paths, both with color sequence $3_62_34_03$;
\item the right and left downward 4-paths, both with color sequence $3_50_36_13$;
\item the lower-right and upper-left downward 3-paths, with common color sequence $3_21_35_43$.
\end{enumerate}
\end{remark}

$(M'_\nabla,d_\nabla)$ is also represented as in the three images on the left half of Figure~\ref{hubo}. The top of these three images bears de vertex-and-face color numbers and edge colors as in the right of Figure~\ref{tubo} on say the upper side of the image of a toroid, also represented schematically with full face colors downward to the left, and on the hidden side of the toroid further to the right.

\section{SETFCs based on the trihexagonal tiling}\label{hex}

In this section, we determine SETFCs based on the trihexagonal tiling \cite{Branko} of $\mathbb{R}^2$.

\begin{theorem}
There are SETFCs in plane and toroidal maps based on the trihexagonal tiling of $\mathbb{R}^2$ via rectangular tiles that are the interior-disjoint union of translated copies of $[0,14]\times[0,2]$, each such copy taken as a map obtained from $P_{15}\square P_3$ by removing, in each one of the seven contiguous copies of $P_3\square P_3$ in it, the copy of $K_{1,4}$ interior to $P_3\square P_3$ and adding two main-diagonal edges, resulting into a submap composed by two triangles and a 6-cycle separating them.  
The resulting maps have their 1-skeletons as the line graphs of those 1-skeletons resulting from Theorem~\ref{tnabla}. 
\end{theorem}

\begin{proof}
The trihexagonal tiling $\mathcal{L}_{\Join}$ is obtained as the line graph of the hexagonal tiling graph treated in Section~\ref{ss3}. It can be furnished as a map 
$M_{\Join}$ with its sets of vertices, edges and faces, that is the elements of $V(M_{\Join})$, $E(M_{\Join})$ and $F(M_{\Join})$, respectively, disposed as in the graph in Figure~\ref{tubo}, which is obtained from the right of Figure~\ref{cubo} by the deletion of certain edges. An SETFC $c_{\Join}$ on $M_{\Join}$ is indicated via the edge colors and vertex and face color numbers (both for hexagons and triangles), obtained from the corresponding ones of $c_\Delta$ on the right of Figure~\ref{cubo}. In fact, every colored hexagon or triangle case of $c_{\Join}$ appears on the right of Figure~\ref{hubo}, representing a toroidal cutout of $(M_\Join,c_\Join)$ from which a representation of $(M'_\Join,d_\Join)$ is obtained by identification of the top and bottom left-to-right 3-paths on the cutout border with common color sequence $1_52_63_04$ as well as identification of the 4-paths following clockwise from those, namely with common color sequence $4_20_53_16_01$ and as well as identification of the remaining 3-paths, namely with common color sequence $1_53_45_14$. 

\begin{figure}[htp]
\includegraphics[scale=0.76]{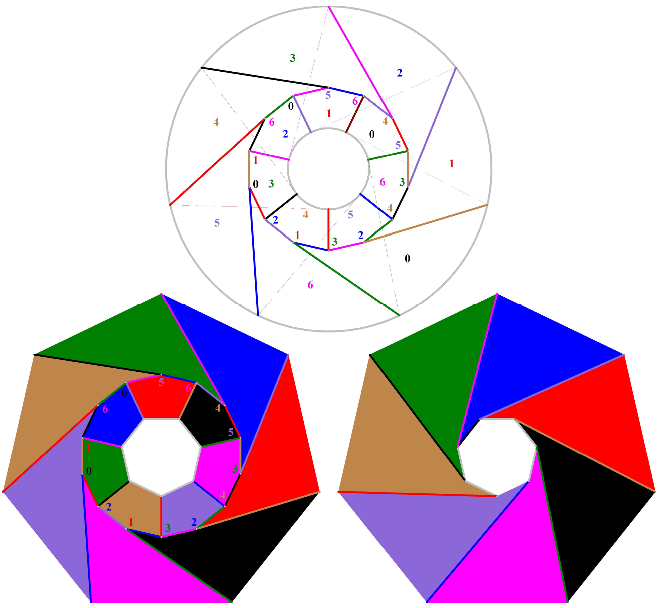}
\includegraphics[scale=0.68]{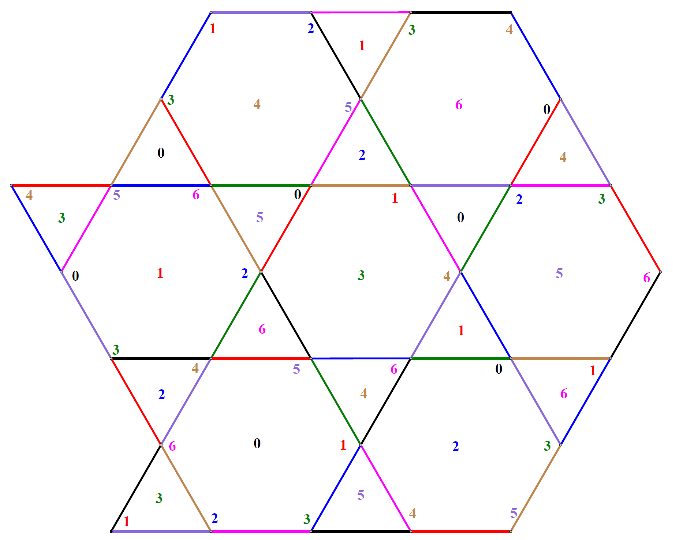}
\caption{SETFCs based on the hexagonal and trihexagonal  (C\&R $(3.6)^2$) tilings.}
\label{hubo}
\end{figure}

Display (\ref{dec27}) encodes this SETFC on a representation of a tile of $\mathcal{L}_{\Join}$ with the smallest number of faces, namely in 7 hexagons and 14 triangles (7 shaped  like $\Delta$ and 7 like $\nabla$). A brace underneath (resp. on top) three contiguous columns in the display represents a triangle $T$ with lower (resp. upper) horizontal edge $e$ and top (resp. bottom) opposite vertex $v$ indicated  in the upper (resp. lower) row with the color-number pair of the hexagonal faces at the left and right  of $v$ bordered on both sides, with the color number $[c_{\Join}(v)]$ between square brackets. If $e'$ and $e''$ are the tilted edges of $T$ and $f$ is its (triangular) face, then  their color numbers appear as ``$(c_{\Join}(e'))[c_{\Join}(f)](c_{\Join}(c''))$''  in the middle row, and if $e=uw$, the lower (resp. upper) row has the remaining colors of $T$ as ``$[c_{\Join}(u)](c_{\Join}(e))[c_{\Join}(w)]$.  
\end{proof}

\begin{eqnarray}\label{dec27}\begin{array}{c}
[4]35\overbrace{[4](1)[5]}46\overbrace{[5](2)[6]}50\overbrace{[6](3)[0]}61\overbrace{[0](4)[1]}02\overbrace{[1](5)[2]}13\overbrace{[2](6)[3]}24\overbrace{[3](0)[4]}\\
(5)[1](2)\,[3](6)[2]\,(3)[4](0)\,[3](4)[5](1)[4](5)[6](2)[5](6)[0](3)[6](0)[1](4)[0](1)[2](5)\\
\underbrace{[6](3)[0]}61\underbrace{[0](4)[1]}02\underbrace{[1](5)[2]}13\underbrace{[2](6)[3]}24\underbrace{[3](0)[4]}35\underbrace{[4](1)[5]}46\underbrace{[5](2)[6]}50[6]\\
\end{array}\end{eqnarray}

\section{Other derivations from the triangular tiling}\label{derived}

\begin{figure}[htp]
\includegraphics[scale=0.71]{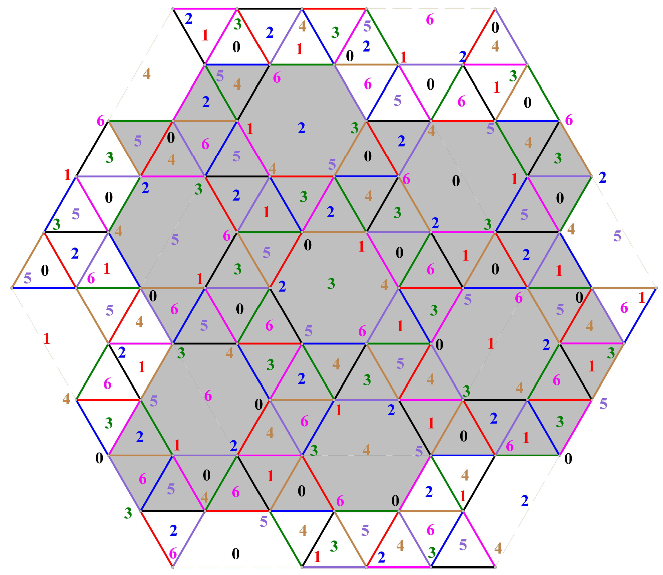}
\includegraphics[scale=0.72]{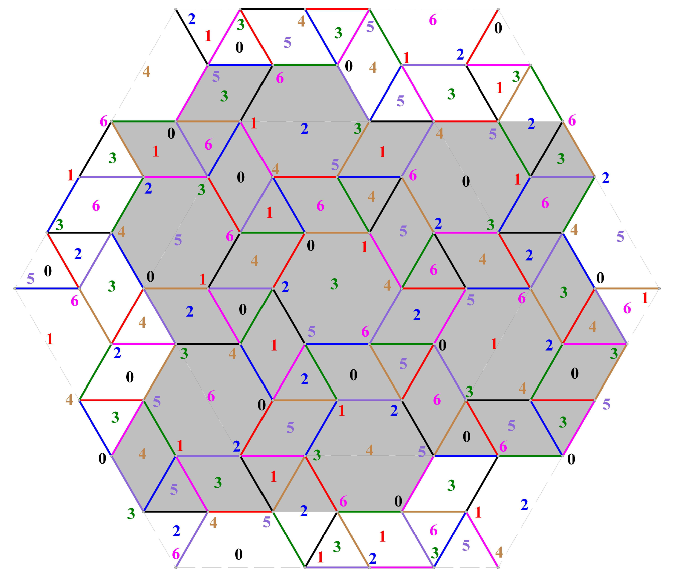}
\caption{ETFC for C\&R $3^4.6$ and SETFC for C\&R 3.4.6.4.}
\label{3464}
\end{figure}

\begin{corollary}
By removing from $(M_\Delta,c_\Delta)$ in the proof of Theorem~\ref{tdelta} an EDS $S'$, a pair $(M''_\delta,f_\delta)$ is obtained in $\mathbb{R}^2$, where $M''_\delta$ is a map based on the {\rm C\&R} $3^4.6$ lattice and $f_\delta$ is an ETFC. By still removing from $(M''_\delta,f_\delta)$ a 1-factor $M''''$ formed by the edges adjacent to any two triangles sharing other edges with corresponding hexagons, another pair $(W_\epsilon,g_\epsilon)$ is obtained, where $W_\epsilon$ is a map based on the {\rm C\&R} $3.4.6.4$ and $g_\epsilon$ is an SETFC. These pairs give place to corresponding toroidal maps. The smallest such toroidal pair from $(M''_\delta,f_\delta)$ has 42 degree-5 vertices, 210 edges, seven hexagonal faces sharing one edge each with one of 42 triangles faces, and 14 additional triangles. The smallest such toroidal pair from $(W_\delta,g_\delta)$ has 42 degree-4 vertices, 210 edges, seven hexagonal faces sharing one edge each with one of 21 rhombic faces, and 14 triangles.  
\end{corollary}

\begin{proof}
Following the left of Figure~\ref{3464}, the shaded area represents a toroidal cutout of $(M''_\delta,f_\delta)$, where the vertices removed from the triangular tiling to form the hexagon faces of the C\&R $3^4.6$ lattice form a set of minimal distance 3 such that any other vertex is a distance one from it, thus constituting an EDS. Then, the first assertion of the statement follows, where each (new) hexagon face receives the color of its central deleted vertex. Rhombic 4-cycles faces are then formed by the remotion of the edges mentioned in the second assertion of the statement, resulting in the pair $(W_\epsilon,g_\epsilon)$, where each (new) rhombic face receives the color of the corresponding deleted edge. Periodic toroidal projections then provide the last assertions of the statement. 
\end{proof}

\begin{figure}[htp]
\includegraphics[scale=0.55]{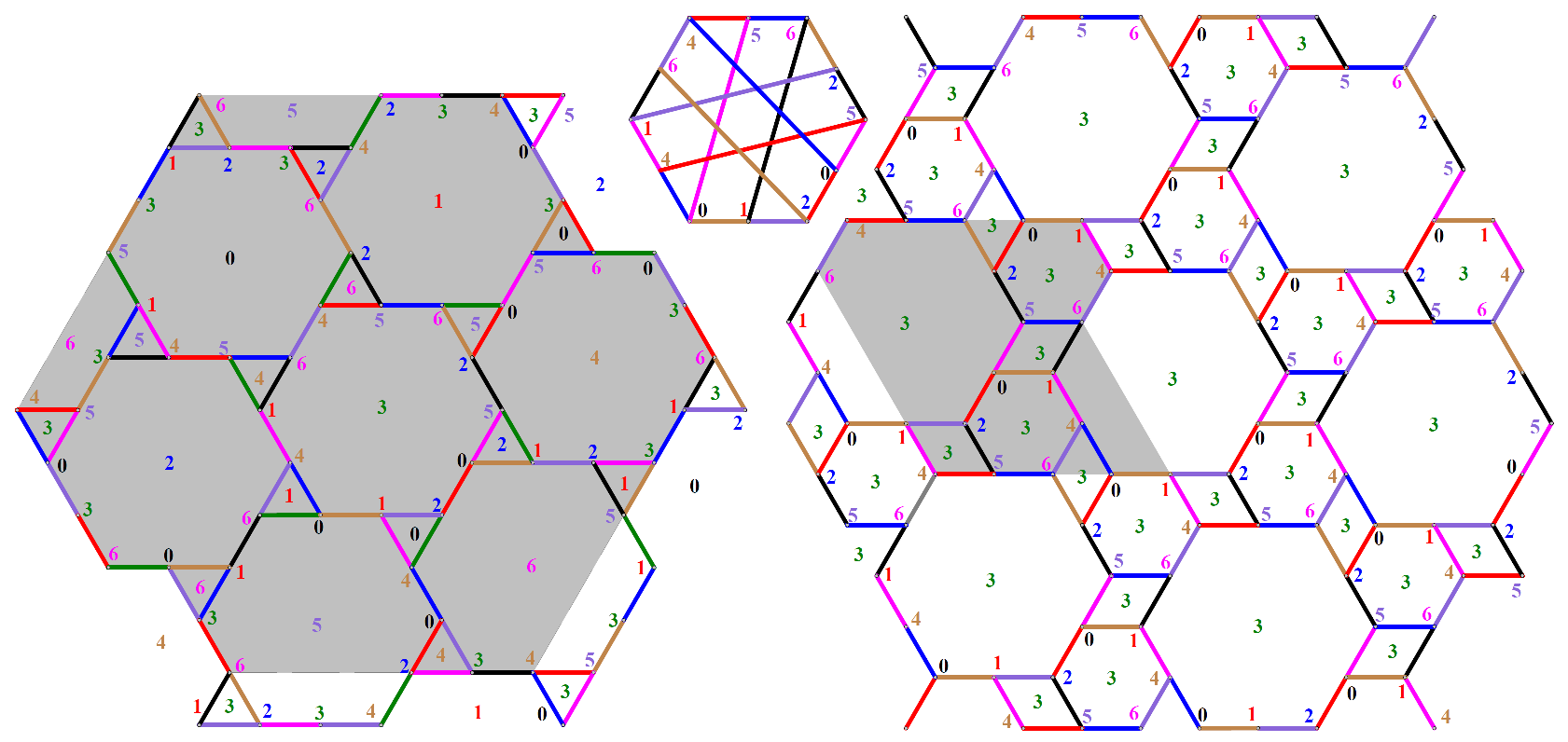}
\caption{Restrictions of $(M_\Delta,c_\Delta)$ for C\&R $3.12^2$ and C\&R 4.6.12 lattices.}
\label{3.12^3}
\end{figure}

\begin{corollary}
By removing from $(M_\Delta,c_\Delta)$ in  Theorem~\ref{tdelta} a PDS $S''$ formed by the closed neighborhoods of the vertices of a sub-lattice of the triangular tiling
of minimal graph distance 4 
generated by the vectors $(3,1)$ and $(1,-3)$ in the coordinate system of its proof, a pair $(\Upsilon_\phi,\upsilon_\phi)$ is obtained in $\mathbb{R}^2$, where $\Upsilon_\phi$ is a map based on the {\rm C\&R} $3.12^2$ lattice and $\upsilon_\phi$ is an SETFC. This pair yields a corresponding toroidal map and SETFCs, the smallest of which with an SETFC that has 42 vertices, 63 edges, 7 dodecagonal faces and 14 triangles.
\end{corollary}

\begin{proof} The statement is illustrated on the left of Figure~\ref{3.12^3}, where a toroidal cutout of the smallest mentioned toroidal map is represented by the shaded area.                         
\end{proof}

\begin{corollary}\label{sma}
Let us remove  from $(M_\Delta,c_\Delta)$ in  Theorem~\ref{tdelta} both: {\bf(a)} a PDS formed by the vertices at graph distances $\le 2$ from a sub-lattice of the triangular tiling of minimal graph distance 5
generated by the vectors $(1,4)$ and $(-1,4)$ in the coordinate system of its proof and {\bf(b)} a 1-factor formed by the edges shared by the remaining pairs of neighboring triangles. Such removals yield a pair $(U_\psi,u_\psi)$, where $U_\psi$ is a map based on the {\rm C\&R} $4.6.12$ lattice and $u_\psi$ is restriction of an ETC (but it is not an ETFC) with opposite vertex or edge pairs in each belt contributing just color pairs $\{0,6\},\{1,5\},\{2,4\}$. The pair $(U_\psi,u_\psi)$ yields a toroidal map with its corresponding restriction of an ETC. The smallest such toroidal map has 12 vertices,18 edges, one dodecagonal face, two hexagonal faces and three square faces. 
\end{corollary}

\begin{proof}
The statement is illustrated on the right of Figure~\ref{3.12^3}, where a cutout of the mentioned smallest toroidal map is represented by the shaded area, the 1-skeleton of which is drawn on the upper center of the figure.
\end{proof}

\begin{conjecture}
The smallest toroidal map of Corollary~\ref{sma} is the smallest toroidal map admitting the restriction of an ETC but not an ETFC.
\end{conjecture}

\begin{figure}[htp]
\includegraphics[scale=0.72]{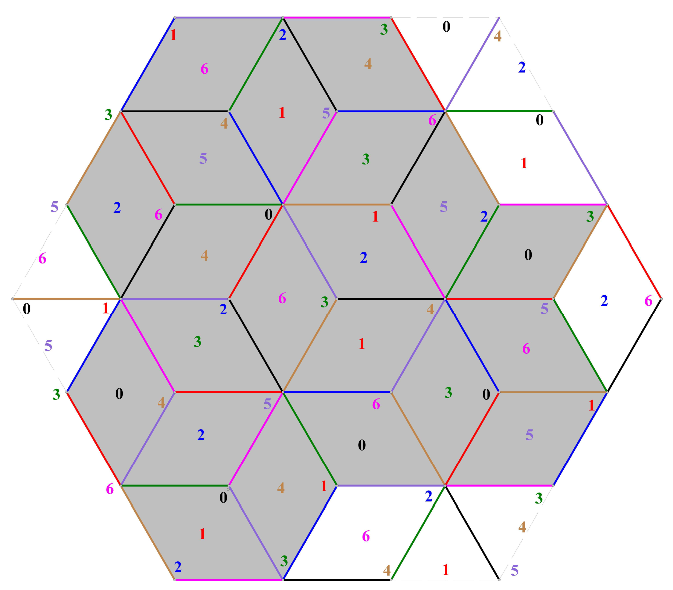}
\includegraphics[scale=0.72]{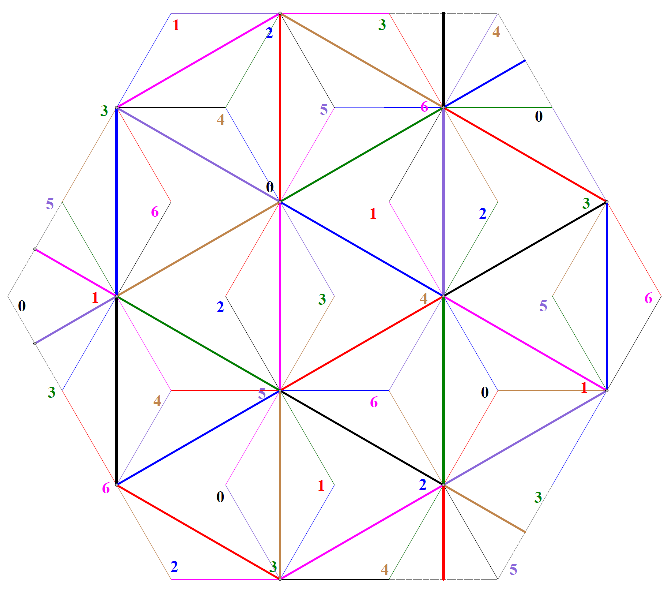}
\caption{SETFC based on the rhombille tiling and recovering $(M_\Delta,c_\Delta)$ from it.}
\label{rombos}
\end{figure}

\begin{corollary} Let $\mathbb{L}$ be the vertex subset of the graph $G=\sigma(M_\Delta)$ of the triangular map $M_\Delta$ containing the origin $(0,0)$ and having minimum graph distance 2. Then,
the graph obtained by removing from $G$ the edges $\xi$ between vertices not in $\mathbb{L}$, 
is the 1-skeleton of a planar map $M_\otimes$ based on the rhombille tiling of the plane whose faces are the unions each of two triangles of $M_\Delta$ adjacent via one such $\xi$.  
Moreover, $M_\otimes$ has an SETFC $c_\otimes$ with $k=7$. 
The pair $(M_\otimes,c_\otimes)$ yields corresponding toroidal maps with associated SETFCs, the smallest such pair $(M'_\otimes,d_\otimes)$ having 7 degree-6 vertices, 14 degree-3 vertices, 42 edges (14 in each of  three edge directions) and 21 rhombic faces.
\end{corollary}

\begin{proof} The statement is illustrated on the left of Figure~\ref{rombos}, where the shaded area covers a toroidal cutout of the pair $(M'_\otimes,d_\otimes)$. The right of the figure shows how to recover $(M_\Delta,c_\Delta)$ from it by considering the long diagonals of the rhombic faces of $M_\otimes$.
\end{proof}

\end{document}